\documentclass[a4paper,reqno]{amsart}

\usepackage{amsmath}
\usepackage{amssymb}
\usepackage{amsthm}
\usepackage{ascmac}
\usepackage{cases}
\usepackage{multicol}
\usepackage[all]{xy}
\usepackage{nccmath}
\usepackage{graphicx}
\usepackage{amscd}
\usepackage[all]{xy}
\usepackage{float}
\usepackage{bm} 
\usepackage{color}
\usepackage{enumerate}
\usepackage{lineno}

\newtheorem{theorem}{Theorem}[section]
\newtheorem{proposition}[theorem]{Proposition}

\newtheorem{lemma}[theorem]{Lemma}
\theoremstyle{definition}
\newtheorem{definition}[theorem]{Definition}

\newtheorem{example}[theorem]{Example}
\newtheorem*{remark*}{Remark}

\newcommand{\type}{\operatorname{type}}

\newcommand{\tri}{\triangleleft}
\newcommand{\pr}{\operatorname{pr}}
\newcommand{\z}{\mathbb{Z}}

\begin{document}

\title[Linear extensions of MCQs and MCQ Alexander pairs]
{Linear extensions of multiple conjugation quandles and MCQ Alexander pairs}

\author[T.~Murao]{Tomo Murao}
\address[T.~Murao]{Institute of Mathematics, University of Tsukuba, Ibaraki 305-8571, Japan}
\email{t-murao@math.tsukuba.ac.jp}


\begin{abstract}
A quandle is an algebra whose axioms are motivated from knot theory.
A linear extension of a quandle can be described by using a pair of maps called an Alexander pair.
In this paper, 
we show that 
a linear extension of a multiple conjugation quandle can be described by using a pair of maps called an MCQ Alexander pair, 
where a multiple conjugation quandle is an algebra whose axioms are motivated from handlebody-knot theory.
\end{abstract}

\maketitle

\section{Introduction}

A quandle~\cite{Joyce82,Matveev82} is an algebra 
whose axioms correspond to the Reidemeister moves for knots.
Andruskiewitsch and Gra\~na~\cite{AndruskiewitschGrana03} introduced a dynamical cocycle 
to construct an extension of a quandle.
Ishii and Oshiro~\cite{IshiiOshiro19} introduced a pair of maps called an Alexander pair, 
which is a dynamical cocycle corresponding to a linear extension of a quandle. 
A linear/affine extension of a quandle plays an important role in constructing knot invariants.
For example, (twisted) Alexander invariants~\cite{Alexander28,Lin01,Wada94} 
and quandle cocycle invariants~\cite{CarterJelsovskyKamadaLangfordSaito03} 
for knots are obtained 
through the theory of quandle extensions (see \cite{IshiiOshiro19}).

A multiple conjugation quandle (MCQ)~\cite{Ishii15} is an algebra 
whose axioms correspond to the Reidemeister moves~\cite{Ishii15(2)} for handlebody-knots.
A handlebody-knot \cite{Ishii08} is a handlebody embedded in the 3-sphere $S^3$, 
which we regard as a generalization of a knot with respect to a genus.
In this paper, 
we introduce a pair of maps called an MCQ Alexander pair.
An MCQ Alexander pair is an MCQ version of an Alexander pair, 
which yields a linear extension of an MCQ.
As with quandles, 
a linear extension of an MCQ plays an important role to construct handlebody-knot invariants.
Actually, 
we can obtain some handlebody-knot invariants by using MCQ Alexander pairs (\cite{IshiiMurao19}).
In this paper, 
we investigate a linear extension of an MCQ with a quadruple of maps.

Any linear extension of a quandle can be realized by using a pair of maps, an Alexander pair.
On the other hand, 
an MCQ is a quandle consisting of a union of groups 
with the quandle operation restricting to conjugation on each group component.
Then any linear extension of an MCQ can be realized by using a quadruple of maps, an ``Alexander quadruple".
However, 
the quadruple has a complicated structure, 
and it is not easy to be handled.
In this paper, 
we show that 
the quadruple of maps which gives a linear extension of an MCQ can be reduced
to some MCQ Alexander pair modulo isomorphism.
That is, 
any linear extension of an MCQ 
can be realized using some MCQ Alexander pair 
up to isomorphism.

This paper is organized into four sections.
In Section~\ref{sec:Multiple conjugation quandles and MCQ Alexander pairs}, 
we recall the notion of a multiple conjugation quandle (MCQ) and introduce an MCQ Alexander pair.
We see that it is related to an extension of an MCQ.
In Section~\ref{sec:Linear extensions of multiple conjugation quandles}, 
we consider linear extensions of MCQs.
We give a quadruple of maps 
corresponding to a linear extension of an MCQ.
In Section~\ref{sec:The reduction of linear extensions of MCQs to MCQ Alexander pairs},
we show that 
any linear extension of an MCQ 
can be realized by using an MCQ Alexander pair 
up to isomorphism.

\section{Multiple conjugation quandles and MCQ Alexander pairs}
\label{sec:Multiple conjugation quandles and MCQ Alexander pairs}

A \textit{quandle}~\cite{Joyce82,Matveev82} is a non-empty set $Q$ 
equipped with a binary operation $\triangleleft:Q\times Q\to Q$ satisfying the following axioms:
\begin{itemize}
\item[(Q1)]
For any $a\in Q$, $a\triangleleft a=a$.
\item[(Q2)]
For any $a\in Q$, the map $\triangleleft a:Q\to Q$ defined by $\triangleleft a(x)=x\triangleleft a$ is bijective.
\item[(Q3)]
For any $a,b,c\in Q$, $(a\triangleleft b)\triangleleft c=(a\triangleleft c)\triangleleft(b\triangleleft c)$.
\end{itemize}
We denote $(\triangleleft a)^n:Q\to Q$ by $\triangleleft^n a$ for $n\in\mathbb{Z}$.
In the following, 
we see some examples of quandles.

\begin{example}
\begin{enumerate}
\item
Let $G$ be a group.
We define a binary operation $\triangleleft$ on $G$ by $a\triangleleft b:=b^{-1}ab$.
Then, $(G,\triangleleft)$ is a quandle.
We call it the \textit{conjugation quandle} of $G$ and denote it by $\operatorname{Conj}G$.

\item
For a positive integer $n$, we denote by $\mathbb{Z}_n$ the cyclic group $\mathbb{Z}/n\mathbb{Z}$ of order $n$.
We define a binary operation $\triangleleft$ on $\mathbb{Z}_n$ by $a\triangleleft b:=2b-a$.
Then, $(\mathbb{Z}_n,\triangleleft)$ is a quandle.
We call it the \textit{dihedral quandle} of order $n$ and denote it by $R_n$.

\item
Let $Q$ be an $R[t^{\pm 1}]$-module, where $R$ is a commutative ring. 
For any $a,b \in Q$, we define a binary operation $\triangleleft$ on $Q$ by $a \tri b := ta+(1-t)b$.
Then $Q$ is a quandle, called an \textit{Alexander quandle}.
\end{enumerate}
\end{example}

For quandles $(Q_1,\triangleleft_1)$ and $(Q_2,\triangleleft_2)$, 
a \textit{quandle homomorphism} $f:Q_1\to Q_2$ is defined 
to be a map $f:Q_1\to Q_2$ satisfying $f(a\triangleleft_1 b)=f(a)\triangleleft_2 f(b)$ for any $a,b\in Q_1$.
We call a bijective quandle homomorphism an \textit{quandle isomorphism}.
$Q_1$ and $Q_2$ are \textit{isomorphic}, denoted $Q_1 \cong Q_2$, 
if there exists an quandle isomorphism from $Q_1$ to $Q_2$.

We define the \textit{type} of a quandle $Q$, denoted $\type Q$, by 
\[\type Q:=\min \{ n \in \mathbb{Z}_{>0} \mid a \tri^n b=a ~(\text{for any~} a,b \in Q) \},\]
where we set $\min \emptyset :=\infty$ for the empty set $\emptyset$.
We note that $(Q,\tri^i)$ is also a quandle for any $i \in \mathbb{Z}$, 
and any finite quandle is of finite type.
For a quandle $Q$, 
an \textit{extension} of $Q$ is a quandle $\widetilde{Q}$ 
which has a surjective homomorphism $f : \widetilde{Q} \to Q$ 
such that for any element of $Q$, 
the cardinality of the inverse image by $f$ is constant.
See also \cite{ElhamdadiNelson15,Nosaka17} for more details on quandles.

\begin{definition}[\cite{Ishii15}]
A \textit{multiple conjugation quandle (MCQ)} $X$ 
is a disjoint union of groups $G_{\lambda} (\lambda \in \Lambda)$ 
with a binary operation $\tri : X \times X \to X$ 
satisfying the following axioms:
\begin{itemize}
\item
For any $a,b \in G_\lambda$, 
$a\tri b = b^{-1}ab$.
\item
For any $x \in X$ and $a,b \in G_\lambda$, 
$x \tri e_\lambda = x$ and $x \tri (ab)=(x \tri a) \tri b$, 
where $e_\lambda$ is the identity of $G_\lambda$.
\item
For any $x,y,z \in X$, 
$(x \tri y) \tri z=(x \tri z) \tri (y \tri z)$.
\item
For any $x \in X$ and $a,b \in G_\lambda$, 
$(ab) \tri x=(a \tri x)(b \tri x)$, 
where $a \tri x, b \tri x \in G_\mu$ for some $\mu \in \Lambda$.
\end{itemize}
\end{definition}

In this paper, we often omit brackets.
When doing so, we apply binary operations from left on expressions, 
except for group operations, 
which we always apply first.
For example, 
we write $a \triangleleft_1 b \triangleleft_2 cd \triangleleft_3 (e \triangleleft_4 f \triangleleft_5 g)$ 
for $((a \triangleleft_1 b) \triangleleft_2 (cd)) \triangleleft_3 ((e \triangleleft_4 f) \triangleleft_5 g)$ simply, 
where each $\triangleleft_i$ is a binary operation, 
and $c$ and $d$ are elements of the same group.
Throughout this paper, 
unless otherwise specified, 
we assume that each $G_\lambda$ is a group 
when $\bigsqcup_{\lambda \in \Lambda}G_\lambda$ is an MCQ.
We denote by $G_a$ the group $G_\lambda$ containing $a \in X$.
We also denote by $e_\lambda$ the identity of $G_\lambda$.
Then the identity of $G_a$ is denoted by $e_a$ for any $a \in X$.

We remark that an MCQ itself is a quandle.
For two MCQs $X_1=\bigsqcup_{\lambda \in \Lambda}G_\lambda$ and $X_2=\bigsqcup_{\mu \in M}G_\mu$, 
an \textit{MCQ homomorphism} $f : X_1 \to X_2$ is defined to be a map 
from $X_1$ to $X_2$ satisfying $f(x \tri y)=f(x) \tri f(y)$ for any $x,y \in X_1$ 
and $f(ab)=f(a)f(b)$ for any $\lambda \in \Lambda$ and $a,b \in G_\lambda$.
We call a bijective MCQ homomorphism an \textit{MCQ isomorphism}.
$X_1$ and $X_2$ are \textit{isomorphic}, denoted $X_1 \cong X_2$, 
if there exists an MCQ isomorphism from $X_1$ to $X_2$.

For an MCQ $X=\bigsqcup_{\lambda \in \Lambda}G_\lambda$, 
an \textit{extension} of $X$ is an MCQ $\widetilde{X}$ 
which has a surjective MCQ homomorphism $f: \widetilde{X} \to X$
such that 
for any element of $X$, 
the cardinality of the inverse image by $f$ is constant.
Then we have the following proposition.

\begin{proposition}\label{prop:MCQ extension}
Let  $X=\bigsqcup_{\lambda \in \Lambda}G_\lambda$ and $\widetilde{X}$ be MCQs.
Then $\widetilde{X}$ is an extension of $X$ 
if and only if 
there exists a set $A$ 
such that 
$\bigcup_{\lambda \in \Lambda}(G_\lambda \times A)$ is an MCQ which is isomorphic to $\widetilde{X}$, 
and that the projection $\pr_X : \bigcup_{\lambda \in \Lambda}(G_\lambda \times A) \to X$ sending $(x,u)$ to $x$ 
is an MCQ homomorphism.
\end{proposition}

\begin{proof}
Assume that $\widetilde{X}$ is an extension of an MCQ $X=\bigsqcup_{\lambda \in \Lambda}G_\lambda$.
There exists a surjective MCQ homomorphism $f: \widetilde{X} \to X$
such that 
for any $x,y \in X$, the cardinalities of $f^{-1}(x)$ and $f^{-1}(y)$ coincide.
Here we note that $\widetilde{X}=\bigcup_{x \in X}f^{-1}(x)$.
Fix $x_0 \in X$ and put $\widetilde{X}=\bigsqcup_{\mu \in M}H_\mu$, 
where $H_\mu$ is a group for each $\mu \in M$.
For any $x\in X$, there exists a bijective map $\phi_x:f^{-1}(x) \to f^{-1}(x_0)$.
We define the map $\phi : \widetilde{X} \to \bigcup_{\lambda \in \Lambda}(G_\lambda \times f^{-1}(x_0))$ by 
$\phi(w):=(f(w),\phi_{f(w)}(w))$.
It is easy to see that $\phi$ is a bijection.
Hence $\bigcup_{\lambda \in \Lambda}(G_\lambda \times f^{-1}(x_0))=\bigsqcup_{\mu \in M}\phi(H_\mu)$ is an MCQ 
with $a \tri b :=\phi(\phi^{-1}(a) \tri \phi^{-1}(b))$ for any $a,b \in \bigcup_{\lambda \in \Lambda}(G_\lambda \times f^{-1}(x_0))$ 
and $ab:=\phi(\phi^{-1}(a) \phi^{-1}(b))$ for any $a,b \in \phi(H_\mu)$.
Then $\phi$ is clearly an MCQ isomorphism.
For the projection $\pr_X : \bigcup_{\lambda \in \Lambda}(G_\lambda \times f^{-1}(x_0)) \to X$ sending $(x,u)$ to $x$, 
it follows $\pr_X=f \circ \phi^{-1}$, 
which implies that $\pr_X$ is an MCQ homomorphism.

Conversely, 
assume that 
there exists a set $A$ such that 
$\bigcup_{\lambda \in \Lambda}(G_\lambda \times A)$ is an MCQ which is isomorphic to $\widetilde{X}$, 
and that the projection $\pr_X : \bigcup_{\lambda \in \Lambda}(G_\lambda \times A) \to X$ sending $(x,u)$ to $x$ 
is an MCQ homomorphism.
There exists an MCQ isomorphism $\phi: \widetilde{X} \to \bigcup_{\lambda \in \Lambda}(G_\lambda \times A)$.
We put $f:=\pr_X \circ \phi$.
Then $f$ is a surjective MCQ homomorphism from $\widetilde{X}$ to $X$ 
such that 
for any element of $X$, 
the cardinality of the inverse image by $f$ is constant.
Therefore $\widetilde{X}$ is an extension of $X$.
\end{proof}

Next, 
we recall the definition of a $G$-family of quandles, 
which is an algebraic system yielding an MCQ.

\begin{definition}[\cite{IshiiIwakiriJangOshiro13}]
Let $G$ be a group with the identity $e$.
A \textit{$G$-family of quandles} is a non-empty set $X$ 
with a family of binary operations $ \tri ^g:X \times X \to X ~(g \in G)$ 
satisfying the following axioms: 
\begin{itemize}
\item
For any $x \in X$ and $g \in G$, 
$x \tri ^gx=x.$
\item
For any $x,y \in X$ and $g,h \in G$, 
$x \tri ^ey=x$ and $x \tri ^{gh}y=(x \tri ^gy) \tri ^hy$.
\item
For any $x,y,z \in X$ and $g,h \in G$, 
$(x \tri ^gy) \tri ^hz=(x \tri ^hz) \tri ^{h^{-1}gh}(y \tri ^hz)$.
\end{itemize}
\end{definition}

Let $R$ be a ring and $G$ be a group with the identity $e$.
Let $X$ be a right $R[G]$-module, where $R[G]$ is the group ring of $G$ over $R$.
Then $(X,\{  \tri ^g \}_{g \in G})$ is a $G$-family of quandles, called a \textit{$G$-family of Alexander quandles}, 
with $x \tri ^gy=xg+y(e-g)$~\cite{IshiiIwakiriJangOshiro13}.
Let $(X, \tri )$ be a quandle 
and put $k:=\type X$.
Then $(X,\{ \tri ^i\}_{i \in \mathbb{Z}_{k}})$ is a $\mathbb{Z}_{k}$-family of quandles, 
where we put $\mathbb{Z}_\infty:=\mathbb{Z}$.
In particular, 
when $X$ is an Alexander quandle, 
$(X,\{ \tri ^i\}_{i \in \mathbb{Z}_{k}})$ is called a \textit{$\mathbb{Z}_{k}$-family of Alexander quandles}.

Let $(X,\{  \tri ^g \}_{g \in G})$ be a $G$-family of quandles.
Then $G \times X =\bigsqcup_{x \in X} (G \times \{ x \})$ is an MCQ with 
\begin{align*}
(g,x) \tri (h,y):=(h^{-1}gh,x \tri ^hy),
\hspace{7mm}
(g,x)(h,x):=(gh,x)
\end{align*}
for any $x,y \in X$ and $g,h \in G$~\cite{Ishii15}.
We call it the \textit{associated MCQ} 
of  $(X,\{  \tri ^g \}_{g \in G})$.
The associated MCQ $G \times X$ of a $G$-family of quandles $X$ 
is an extension of an MCQ~$G$.

Throughout this paper, unless otherwise stated, 
we assume that every ring has the multiplicative identity $1 \neq 0$.
For a ring $R$, 
we denote by $R^\times$ the group of units of $R$.
In the following, 
we introduce a pair of maps, called an MCQ Alexander pair, 
which corresponds to a linear extension of an MCQ 
as seen in Proposition~\ref{prop:MCQ Alexander pair}.

\begin{definition}
Let $X=\bigsqcup_{\lambda \in \Lambda}G_\lambda$ be an MCQ 
and $R$ a ring.
The pair $(f_1,f_2)$ of maps $f_1,f_2 : X \times X \to R$ 
is an \textit{MCQ Alexander pair} 
if $f_1$ and $f_2$ satisfy the following conditions: 

\begin{itemize} 
\item
For any $a,b \in G_\lambda$,
\begin{align*}
f_1(a,b)+f_2(a,b)=f_1(a,a^{-1}b).
\end{align*}

\item
For any $a,b \in G_\lambda$ and $x \in X$, 
\begin{align*}
&f_1(a,x)=f_1(b,x),\\
&f_2(ab,x)=f_2(a,x)+f_1(b \tri x,a^{-1} \tri x)f_2(b,x).
\end{align*}

\item
For any $x \in X$ and $a,b \in G_{\lambda}$,
\begin{align*}
&f_1(x,e_{\lambda})=1,\\
&f_1(x,ab)=f_1(x \tri a,b)f_1(x,a),\\
&f_2(x,ab)=f_1(x \tri a,b)f_2(x,a).
\end{align*}

\item
For any $x,y,z \in X$, 
\begin{align*}
&f_1(x \tri y,z)f_1(x,y)=f_1(x \tri z,y \tri z)f_1(x,z),\\
&f_1(x \tri y,z)f_2(x,y)=f_2(x \tri z,y \tri z)f_1(y,z),\\
&f_2(x \tri y,z)=f_1(x \tri z,y \tri z)f_2(x,z)+f_2(x \tri z,y \tri z)f_2(y,z).
\end{align*}
\end{itemize}
\end{definition}

By the definition of an MCQ Alexander pair, 
we have the following lemma.

\begin{lemma}\label{lem:MCQAlexanderPairProperty}
Let $X=\bigsqcup_{\lambda \in \Lambda}G_\lambda$ be an MCQ and $R$ a ring.
Let $(f_1,f_2)$ be an MCQ Alexander pair of maps $f_1,f_2 : X \times X \to R$.
For any $x,y \in X$ and $a,b \in G_\lambda$, 
the following hold.
\begin{align*}
&\text{$f_1(x,y)$ is invertible, and~} 
f_1(x,y)^{-1}=f_1(x \tri y,y^{-1}),\\
&f_2(e_\lambda,x)=0,\\
&f_1(ab,x)f_1(a,a^{-1})=f_1(b \tri x,a^{-1} \tri x)f_1(b,x),\\
&f_2(x \tri a,b)=f_2(x,ab)f_1(a,a^{-1}).
\end{align*}
\end{lemma}

\begin{proof}
Since for any $x,y,z \in X$ and $a,b \in G_\lambda$, 
\begin{align*}
&f_1(x \tri y,z)f_1(x,y)=f_1(x \tri z,y \tri z)f_1(x,z),\\
&f_1(x,ab)=f_1(x \tri a,b)f_1(x,a),\\
&f_2(ab,x)=f_2(a,x)+f_1(b \tri x,a^{-1} \tri x)f_2(b,x),
\end{align*}
we have that 
$f_1(x,y)^{-1}=f_1(x \tri y,y^{-1})$ and $f_2(e_\lambda,x)=0$.
For any $x\in X$ and $a,b \in G_\lambda$, 
\begin{align*}
f_1(ab,x)f_1(a,a^{-1})
&=f_1(b \tri a^{-1},x)f_1(b,a^{-1})\\
&=f_1(b \tri x,a^{-1} \tri x)f_1(b,x),\\
f_2(x \tri a,b)
&=f_2(x \tri a,aba^{-1} \tri a)f_1(aba^{-1},a)f_1(a,a^{-1})\\
&=f_1(x \tri aba^{-1}, a)f_2(x,aba^{-1})f_1(a,a^{-1})\\
&=f_2(x,ab)f_1(a,a^{-1}).
\end{align*}
\end{proof}

By the definition and Lemma~\ref{lem:MCQAlexanderPairProperty}, 
we can easily check that 
an MCQ Alexander pair is an Alexander pair~\cite{IshiiOshiro19}.
We call $(1,0)$ the \textit{trivial MCQ Alexander pair}, 
where $0$ and $1$ respectively denote the zero map and the constant map 
that sends all elements of the domain to the multiplicative identity $1$ of the ring.
An MCQ Alexander pair corresponds to an extension of an MCQ 
as shown in the following proposition.

\begin{proposition}\label{prop:MCQ Alexander pair}
Let $X=\bigsqcup_{\lambda \in \Lambda}G_\lambda$ be an MCQ and $R$ a ring.
Let $f_1,f_2: X \times X \to R$ be maps.
Then the pair $(f_1,f_2)$ is an MCQ Alexander pair 
if and only if 
$\widetilde{X}(f_1,f_2):=\bigsqcup_{\lambda \in \Lambda}(G_\lambda \times M)$ 
is an MCQ with 
\begin{align*}
& (x,u) \tri (y,v):=(x \tri y, f_1(x,y)u+f_2(x,y)v)& &((x,u),(y,v) \in \widetilde{X}(f_1,f_2)),\\
& (a,u)(b,v):=(ab,u+f_1(a,a^{-1})v)& &((a,u),(b,v) \in G_\lambda \times M)
\end{align*}
for any left $R$-module $M$.
\end{proposition}

\begin{proof}
If $(f_1,f_2)$ is an MCQ Alexander pair, 
then we have that $\widetilde{X}(f_1,f_2)=\bigsqcup_{\lambda \in \Lambda}(G_\lambda \times M)$ 
is an MCQ for any left $R$-module $M$ 
by direct calculation and by Lemma~\ref{lem:MCQAlexanderPairProperty}.
Here, the identity of $G_\lambda \times M$ is $(e_\lambda,0)$, 
and the inverse of $(a,u)$ is $(a^{-1},-f_1(a,a)u)$ 
for any $(a,u) \in G_\lambda \times M$.

Put $M:=R$.
Assume that $\widetilde{X}(f_1,f_2)=\bigsqcup_{\lambda \in \Lambda}(G_\lambda \times M)$ is an MCQ.
Then we prove that $(f_1,f_2)$ is an MCQ Alexander pair.
For each $\lambda \in \Lambda$, $G_\lambda \times M$ is a group.
Hence for any $(a,u), (b,v), (c,w) \in G_\lambda \times M$, it follows that
\begin{align*}
((a,u)(b,v))(c,w)
&= (ab,u+f_1(a,a^{-1})v)(c,w)\\
&= (abc,u+f_1(a,a^{-1})v+f_1(ab,b^{-1}a^{-1})w),\\
(a,u)((b,v)(c,w))
&= (a,u)(bc,v+f_1(b,b^{-1})w)\\
&= (abc,u+f_1(a,a^{-1})(v+f_1(b,b^{-1})w)).
\end{align*}
By the associativity of $G_\lambda \times M$, 
we have that for any $a,b \in G_\lambda$, 
\begin{align}
f_1(a,a^{-1})f_1(b,b^{-1})=f_1(ab,b^{-1}a^{-1}).
\end{align}

For any $(a,u), (b,v) \in G_\lambda \times M$, $(a,u) \tri (b,v)=(b,v)^{-1}(a,u)(b,v)$.
It follows that 
\begin{align*}
(a,u) \tri (b,v)
&= (b^{-1}ab,f_1(a,b)u+f_2(a,b)v),\\
(b,v)^{-1}(a,u)(b,v)
&= (b^{-1},-f_1(b,b)v)(ab,u+f_1(a,a^{-1})v)\\
&= (b^{-1}ab,-f_1(b,b)v+f_1(b^{-1},b)(u+f_1(a,a^{-1})v)).
\end{align*}
Hence we have that for any $a,b \in G_\lambda$, 
\begin{align}
f_1(a,b) &=f_1(b^{-1},b), \notag\\
f_2(a,b) &=-f_1(b,b)+f_1(b^{-1},b)f_1(a,a^{-1}).
\end{align}

For any $(x,u) \in \widetilde{X}(f_1,f_2)$ and $(a,v), (b,w) \in G_{\lambda} \times M$, 
$(x,u) \tri (e_{\lambda},0)=(x,u)$ and $(x,u) \tri ((a,v)(b,w))=((x,u) \tri (a,v)) \tri (b,w)$.
It follows that 
\begin{align*}
(x,u) \tri (e_{\lambda},0) &=(x,f_1(x,e_{\lambda})u),\\
(x,u) \tri ((a,v)(b,w))
&= (x,u) \tri (ab,v+f_1(a,a^{-1})w)\\
&= (x \tri ab,f_1(x,ab)u+f_2(x,ab)(v+f_1(a,a^{-1})w)),\\
((x,u) \tri (a,v)) \tri (b,w)
&= (x \tri a,f_1(x,a)u+f_2(x,a)v) \tri (b,w)\\
&= ((x \tri a) \tri b,f_1(x \tri a,b)(f_1(x,a)u+f_2(x,a)v)+f_2(x \tri a,b)w).
\end{align*}
Hence we have that for any $x \in X$ and $a,b \in G_{\lambda}$, 
\begin{align}
&f_1(x,e_{\lambda}) =1,\\
&f_1(x,ab) = f_1(x \tri a,b)f_1(x,a), \\ 
&f_2(x,ab) = f_1(x \tri a,b)f_2(x,a),\\
&f_2(x,ab)f_1(a,a^{-1})=f_2(x \tri a,b). \notag
\end{align}

For any $(x,u), (y,v), (z,w) \in \widetilde{X}(f_1,f_2)$, 
$((x,u) \tri (y,v)) \tri (z,w)=((x,u) \tri (z,w)) \tri ((y,v) \tri (z,w))$.
It follows that 
\begin{align*}
& ((x,u) \tri (y,v)) \tri (z,w) \\
&= (x \tri y,f_1(x,y)u+f_2(x,y)v) \tri (z,w)\\
&= ((x \tri y) \tri z,f_1(x \tri y,z)(f_1(x,y)u+f_2(x,y)v)+f_2(x \tri y,z)w),\\
& ((x,u) \tri (z,w)) \tri ((y,v) \tri (z,w))\\
&=(x \tri z, f_1(x,z)u+f_2(x,z)w) \tri (y \tri z, f_1(y,z)v+f_2(y,z)w)\\
&=((x \tri z)\tri (y \tri z), f_1(x \tri z, y \tri z)( f_1(x,z)u+f_2(x,z)w)\\
& \quad +f_2(x \tri z, y \tri z)(f_1(y,z)v+f_2(y,z)w)).
\end{align*}
Hence we have that for any $x,y,z \in X$, 
\begin{align}
&f_1(x \tri y,z)f_1(x,y) = f_1(x \tri z, y \tri z)f_1(x,z), \\
&f_1(x \tri y,z)f_2(x,y) = f_2(x \tri z, y \tri z)f_1(y,z), \\
&f_2(x \tri y,z) =  f_1(x \tri z, y \tri z)f_2(x,z)+f_2(x \tri z, y \tri z)f_2(y,z).
\end{align}

For any $(a,u), (b,v) \in G_\lambda \times M$ and $(x,w) \in \widetilde{X}(f_1,f_2)$, 
$((a,u)(b,v)) \tri (x,w)=((a,u) \tri (x,w))((b,v) \tri (x,w))$, 
where we note that 
$(a,u) \tri (x,w), (b,v) \tri (x,w) \in G_{\mu} \times M$ for some $\mu \in \Lambda$.
It follows that 
\begin{align*}
&((a,u)(b,v)) \tri (x,w)\\
&=(ab,u+f_1(a,a^{-1})v) \tri (x,w)\\
&=(ab \tri x,f_1(ab,x)(u+f_1(a,a^{-1})v) + f_2(ab,x)w),\\
&((a,u) \tri (x,w))((b,v) \tri (x,w)) \\
&=(a \tri x, f_1(a,x)u+f_2(a,x)w)(b \tri x, f_1(b,x)v+f_2(b,x)w)\\
&=((a \tri x)(b \tri x),f_1(a,x)u+f_2(a,x)w\\
&\quad +f_1(a \tri x,a^{-1} \tri x)(f_1(b,x)v+f_2(b,x)w)).
\end{align*}
Hence we have that for any $a,b \in G_\lambda$ and $x \in X$, 
\begin{align}
&f_1(ab,x)=f_1(a,x),\\
&f_1(ab,x)f_1(a,a^{-1})=f_1(a \tri x,a^{-1} \tri x)f_1(b,x), \notag \\
&f_2(ab,x)=f_2(a,x)+f_1(a \tri x,a^{-1} \tri x)f_2(b,x).
\end{align}

By equations (1), (2), (9) and (10), 
we have that for any $a,b \in G_\lambda$ and $x \in X$, 
\begin{align}
&f_1(a,b)+f_2(a,b)=f_1(a,a^{-1}b), \tag{2'}\\
&f_1(a,x)=f_1(b,x), \tag{9'}\\
&f_2(ab,x)=f_2(a,x)+f_1(b \tri x,a^{-1} \tri x)f_2(b,x). \tag{10'}
\end{align}
Therefore, 
by the equations (2'), (3)--(8), (9') and (10'), 
the pair $(f_1,f_2)$ is an MCQ Alexander pair.
\end{proof}

We remark that 
the MCQ $\widetilde{X}(f_1,f_2)=\bigsqcup_{\lambda \in \Lambda}(G_\lambda \times M)$ 
in Proposition~\ref{prop:MCQ Alexander pair} is an extension of $X$ 
since the projection $\pr_X:\widetilde{X}(f_1,f_2) \to X$ sending $(x,u)$ to $x$ 
satisfies the defining condition of an extension.

We give some examples of MCQ Alexander pairs.

\begin{example}\label{ex:MCQ Alexander pair of Alexander invariant}
Let $R$ be a ring and $G_0$ be the abelian group 
\[\left\langle t_1,\ldots, t_r \,\middle| \, t_1^{k_1},\ldots, t_r^{k_r}, [t_i,t_j] ~(1 \leq i<j \leq r) \right\rangle\]
for some $k_1, \ldots, k_r \in \z_{\geq 0}$, 
where $[t_i,t_j]$ indicates the commutator of $t_i$ and $t_j$.
We remark that 
the group ring $R[G_0]$ may be identified with the quotient ring of Laurent polynomial ring 
$R[ t_1^{\pm 1},\ldots, t_r^{\pm 1} ]/(t_1^{k_1}-1,\ldots, t_r^{k_r}-1)$.
We set maps $f_1,f_2:G_0 \times G_0 \to R[G_0]$ by 
\begin{align*}
f_1(s,t)=t^{-1},~
f_2(s,t)=t^{-1}s-t^{-1}.
\end{align*}
Then the pair $(f_1,f_2)$ is an MCQ Alexander pair.
\end{example}


\begin{example}
Let $G$ be a group, $R$ a ring and $X=\bigsqcup_{x \in Y}(G \times \{x\})$ 
be an associated MCQ of a $G$-family of quandles $(Y,\{ \tri^g\}_{g \in G})$.
Let $f:G \to G$ be a group homomorphism.
We set maps 
$f_1,f_2:X \times X \to R[G]$
by 
\begin{align*}
f_1((a,x),(b,y))=f(b)^{-1},~
f_2((a,x),(b,y))=f(b)^{-1}(f(a)-1).
\end{align*}
Then the pair $(f_1,f_2)$ is an MCQ Alexander pair.
\end{example}

\section{Linear extensions of multiple conjugation quandles}
\label{sec:Linear extensions of multiple conjugation quandles}

Let $X=\bigsqcup_{\lambda \in \Lambda}G_\lambda$ be an MCQ and $R$ a ring.
Let $f_1,f_2: X \times X \to R$ and $f_3,f_4: \bigsqcup_{\lambda \in \Lambda}(G_\lambda \times G_\lambda) \to R$ be maps.
In this section, 
we consider a linear extension of $X$ using $f_1$, $f_2$, $f_3$ and $f_4$.

We define the conditions \eqref{eq:0-i}--\eqref{eq:4-iii} for $f_1$, $f_2$, $f_3$ and $f_4$ 
as follows:

\begin{itemize}
\item
For any $a,b,c \in G_\lambda$, 
\begin{align}
& \text{$f_3(a,b)$ and $f_4(a,b)$ are invertible}, \tag{0-i} \label{eq:0-i}\\
& f_3(ab,c)f_3(a,b)=f_3(a,bc),\tag{0-ii} \label{eq:0-ii}\\
& f_3(ab,c)f_4(a,b)=f_4(a,bc)f_3(b,c),\tag{0-iii} \label{eq:0-iii}\\
& f_4(ab,c)=f_4(a,bc)f_4(b,c)\tag{0-iv}. \label{eq:0-iv}
\end{align}

\item
For any $a,b \in G_\lambda$, 
\begin{align*}
f_1(a,b) &=f_4(b^{-1},ab)f_3(a,b),\tag{1-i} \label{eq:1-i}\\
f_2(a,b) &=-f_3(b^{-1},ab)f_4(b^{-1},e_\lambda)f_3(b,b^{-1})+f_4(b^{-1},ab)f_4(a,b). \tag{1-ii} \label{eq:1-ii}
\end{align*}

\item
For any $x \in X$ and $a,b \in G_{\lambda}$, 
\begin{align*}
&f_1(x,e_{\lambda}) =1,\tag{2-i} \label{eq:2-i}\\
&f_1(x,ab) = f_1(x \tri a,b)f_1(x,a), \tag{2-ii} \label{eq:2-ii}\\ 
&f_2(x,ab)f_3(a,b) = f_1(x \tri a,b)f_2(x,a), \tag{2-iii} \label{eq:2-iii}\\
&f_2(x,ab)f_4(a,b)=f_2(x \tri a,b). \tag{2-iv} \label{eq:2-iv}
\end{align*}

\item
For any $x,y,z \in X$, 
\begin{align*}
&f_1(x \tri y,z)f_1(x,y) = f_1(x \tri z, y \tri z)f_1(x,z), \tag{3-i} \label{eq:3-i}\\
&f_1(x \tri y,z)f_2(x,y) = f_2(x \tri z, y \tri z)f_1(y,z), \tag{3-ii} \label{eq:3-ii}\\
&f_2(x \tri y,z) =  f_1(x \tri z, y \tri z)f_2(x,z)+f_2(x \tri z, y \tri z)f_2(y,z). \tag{3-iii} \label{eq:3-iii}
\end{align*}

\item
For any $a,b \in G_\lambda$ and $x \in X$, 
\begin{align*}
&f_1(ab,x)f_3(a,b)=f_3(a \tri x,b \tri x)f_1(a,x), \tag{4-i} \label{eq:4-i}\\
&f_1(ab,x)f_4(a,b)=f_4(a \tri x,b \tri x)f_1(b,x), \tag{4-ii} \label{eq:4-ii}\\
&f_2(ab,x)=f_3(a \tri x,b \tri x)f_2(a,x)+f_4(a \tri x,b \tri x)f_2(b,x). \tag{4-iii} \label{eq:4-iii}
\end{align*}
\end{itemize}

These conditions correspond to every linear extension of an MCQ 
as seen in Proposition~\ref{Alexander quadruple}.
We remark that for any $a,b \in G_\lambda$, 
it follows 
\begin{align*}
&f_3(a,e_\lambda)=1=f_4(e_\lambda,a),\\
&f_3(a,b)^{-1}=f_3(ab,b^{-1}),\\
&f_4(a,b)^{-1}=f_4(a^{-1},ab)
\end{align*}
by \eqref{eq:0-i}, \eqref{eq:0-ii} and \eqref{eq:0-iv}.

\begin{proposition}\label{Alexander quadruple}
Let $X=\bigsqcup_{\lambda \in \Lambda}G_\lambda$ be an MCQ and $R$ a ring.
Let $f_1,f_2: X \times X \to R$ and $f_3,f_4: \bigsqcup_{\lambda \in \Lambda}(G_\lambda \times G_\lambda) \to R$ be maps.
Then $f_1$, $f_2$, $f_3$ and $f_4$ satisfy the conditions \eqref{eq:0-i}--\eqref{eq:4-iii} 
if and only if 
$\widetilde{X}(f_1,f_2,f_3,f_4):=\bigsqcup_{\lambda \in \Lambda}(G_\lambda \times M)$ 
is an MCQ with 
\begin{align*}
& (x,u) \tri (y,v):=(x \tri y, f_1(x,y)u+f_2(x,y)v)& &((x,u),(y,v) \in \widetilde{X}(f_1,f_2,f_3,f_4)),\\
& (a,u)(b,v):=(ab,f_3(a,b)u+f_4(a,b)v)& &((a,u),(b,v) \in G_\lambda \times M).
\end{align*}
 for any left $R$-module $M$.
 \end{proposition}

Let us first prove the following lemma in order to prove Proposition~\ref{Alexander quadruple} later.

\begin{lemma}\label{group condition}
In the same situation as Proposition~\ref{Alexander quadruple}, 
the maps $f_3$ and $f_4$ satisfy the conditions \eqref{eq:0-i}--\eqref{eq:0-iv} 
if and only if 
$G_\lambda \times M$ is a group for each $\lambda \in \Lambda$ 
and any left $R$-module $M$.
\end{lemma}

\begin{proof}
If $f_3$ and $f_4$ satisfy the conditions \eqref{eq:0-i}--\eqref{eq:0-iv}, 
then we have that $G_\lambda \times M$ is a group for each $\lambda \in \Lambda$ 
and any left $R$-module $M$ by direct calculation.
Here, the identity of $G_\lambda \times M$ is $(e_\lambda,0)$, 
and the inverse of $(a,u)$ is $(a^{-1},-f_4(a^{-1},e_\lambda)f_3(a,a^{-1})u)$ 
for any $(a,u) \in G_\lambda \times M$.

Put $M:=R$.
Assume that $G_\lambda \times M$ is a group for each $\lambda \in \Lambda$.
Then it follows that 
for any $(a,u),(b,v),(c,w) \in G_\lambda \times M$, 
\begin{align*}
((a,u)(b,v))(c,w)
&= (ab,f_3(a,b)u+f_4(a,b)v)(c,w)\\
&= (abc,f_3(ab,c)(f_3(a,b)u+f_4(a,b)v)+f_4(ab,c)w),\\
(a,u)((b,v)(c,w))
&= (a,u)(bc,f_3(b,c)v+f_4(b,c)w)\\
&= (abc,f_3(a,bc)u+f_4(a,bc)(f_3(b,c)v+f_4(b,c)w)).
\end{align*}
By the associativity of $G_\lambda \times M$, 
we have that 
$f_1,f_2,f_3$ and $f_4$ satisfy the conditions \eqref{eq:0-ii}, \eqref{eq:0-iii} and \eqref{eq:0-iv}.

Let $(g,m)$ be the identity of $G_\lambda \times M$.
Then for any $(a,u) \in G_\lambda \times M$, 
it follows that 
\begin{align*}
(a,u)(g,m)=(ag,f_3(a,g)u+f_4(a,g)m)=(a,u),\\
(g,m)(a,u)=(ga,f_3(g,a)m+f_4(g,a)u)=(a,u).
\end{align*}
Hence we have $g=e_\lambda$ and $f_3(a,e_\lambda)=1=f_4(e_\lambda,a)$ for any $a \in G_\lambda$.
By \eqref{eq:0-ii} and \eqref{eq:0-iv}, 
we obtain that $f_1,f_2,f_3$ and $f_4$ satisfy the condition \eqref{eq:0-i}, 
and that $f_3(a,b)^{-1}=f_3(ab,b^{-1})$ and $f_4(a,b)^{-1}=f_4(a^{-1},ab)$.
%
\end{proof}

\begin{proof}[Proof of Proposition~\ref{Alexander quadruple}]
If $f_1$, $f_2$, $f_3$ and $f_4$ satisfy the conditions \eqref{eq:0-i}--\eqref{eq:4-iii}, 
then we have that 
$\widetilde{X}(f_1,f_2,f_3,f_4)$ is an MCQ for any left $R$-module $M$ 
by direct calculation and by Lemma~\ref{group condition}.

Put $M:=R$.
Assume that 
$\widetilde{X}(f_1,f_2,f_3,f_4)=\bigsqcup_{\lambda \in \Lambda}(G_\lambda \times M)$ is an MCQ.
For each $\lambda \in \Lambda$, 
$G_\lambda \times M$ is a group.
Hence we have that $f_1,f_2,f_3$ and $f_4$ satisfy the conditions \eqref{eq:0-i}--\eqref{eq:0-iv} 
by Lemma~\ref{group condition}.

For any $(a,u), (b,v) \in G_\lambda \times M$, $(a,u) \tri (b,v)=(b,v)^{-1}(a,u)(b,v)$.
It follows that 
\begin{align*}
(a,u) \tri (b,v)
&= (a \tri b,f_1(a,b)u+f_2(a,b)v)\\
&= (b^{-1}ab,f_1(a,b)u+f_2(a,b)v),\\
(b,v)^{-1}(a,u)(b,v)
&= (b^{-1},-f_4(b^{-1},e_\lambda)f_3(b,b^{-1})v)(ab,f_3(a,b)u+f_4(a,b)v)\\
&= (b^{-1}ab,f_3(b^{-1},ab)(-f_4(b^{-1},e_\lambda)f_3(b,b^{-1})v) \\
   &\quad +f_4(b^{-1},ab)(f_3(a,b)u+f_4(a,b)v)).
\end{align*}
Hence we have that $f_1,f_2,f_3$ and $f_4$ satisfy the conditions \eqref{eq:1-i} and \eqref{eq:1-ii}.

For any $(x,u) \in \widetilde{X}(f_1,f_2,f_3,f_4)$ and $(a,v), (b,w) \in G_{\lambda} \times M$, 
$(x,u) \tri (e_{\lambda},0)=(x,u)$ and $(x,u) \tri ((a,v)(b,w))=((x,u) \tri (a,v)) \tri (b,w)$.
It follows that 
\begin{align*}
(x,u) \tri (e_{\lambda},0) &=(x,f_1(x,e_{\lambda})u),\\
(x,u) \tri ((a,v)(b,w))
&= (x,u) \tri (ab,f_3(a,b)v+f_4(a,b)w)\\
&= (x \tri ab,f_1(x,ab)u+f_2(x,ab)(f_3(a,b)v+f_4(a,b)w)),\\
((x,u) \tri (a,v)) \tri (b,w)
&= (x \tri a,f_1(x,a)u+f_2(x,a)v) \tri (b,w)\\
&= ((x \tri a) \tri b,f_1(x \tri a,b)(f_1(x,a)u+f_2(x,a)v)\\
&\quad+f_2(x \tri a,b)w).
\end{align*}
Hence we have that $f_1,f_2,f_3$ and $f_4$ satisfy the conditions \eqref{eq:2-i}--\eqref{eq:2-iv}.

For any $(x,u), (y,v), (z,w) \in \widetilde{X}(f_1,f_2,f_3,f_4)$, 
$((x,u) \tri (y,v)) \tri (z,w)=((x,u) \tri (z,w)) \tri ((y,v) \tri (z,w))$.
It follows that 
\begin{align*}
& ((x,u) \tri (y,v)) \tri (z,w) \\
&= (x \tri y,f_1(x,y)u+f_2(x,y)v) \tri (z,w)\\
&= ((x \tri y) \tri z,f_1(x \tri y,z)(f_1(x,y)u+f_2(x,y)v)+f_2(x \tri y,z)w),\\
& ((x,u) \tri (z,w)) \tri ((y,v) \tri (z,w))\\
&=(x \tri z, f_1(x,z)u+f_2(x,z)w) \tri (y \tri z, f_1(y,z)v+f_2(y,z)w)\\
&=((x \tri z)\tri (y \tri z), f_1(x \tri z, y \tri z)( f_1(x,z)u+f_2(x,z)w)\\
& \quad +f_2(x \tri z, y \tri z)(f_1(y,z)v+f_2(y,z)w)).
\end{align*}
Hence we have that $f_1,f_2,f_3$ and $f_4$ satisfy the conditions \eqref{eq:3-i}--\eqref{eq:3-iii}.

For any $(a,u), (b,v) \in G_\lambda \times M$ and $(x,w) \in \widetilde{X}(f_1,f_2,f_3,f_4)$, 
$((a,u)(b,v)) \tri (x,w)=((a,u) \tri (x,w))((b,v) \tri (x,w))$, 
where we note that 
$(a,u) \tri (x,w), (b,v) \tri (x,w) \in G_{\mu} \times M$ for some $\mu \in \Lambda$.
It follows that 
\begin{align*}
&((a,u)(b,v)) \tri (x,w)\\
&=(ab,f_3(a,b)u+f_4(a,b)v) \tri (x,w)\\
&=(ab \tri x,f_1(ab,x)(f_3(a,b)u+f_4(a,b)v) + f_2(ab,x)w),\\
&((a,u) \tri (x,w))((b,v) \tri (x,w)) \\
&=(a \tri x, f_1(a,x)u+f_2(a,x)w)(b \tri x, f_1(b,x)v+f_2(b,x)w)\\
&=((a \tri x)(b \tri x),f_3(a \tri x,b \tri x)(f_1(a,x)u+f_2(a,x)w)\\
&\quad +f_4(a \tri x,b \tri x)(f_1(b,x)v+f_2(b,x)w)).
\end{align*}
Hence we have that $f_1,f_2,f_3$ and $f_4$ satisfy the conditions \eqref{eq:4-i}--\eqref{eq:4-iii}.

This completes the proof.
\end{proof}

We remark that 
the MCQ $\widetilde{X}(f_1,f_2,f_3,f_4)=\bigsqcup_{\lambda \in \Lambda}(G_\lambda \times M)$ 
in Proposition~\ref{Alexander quadruple} 
is an extension of $X$ 
since the projection $\pr_X:\widetilde{X}(f_1,f_2,f_3,f_4) \to X$ sending $(x,u)$ to $x$ 
satisfies the defining condition of an extension.
We call it a \textit{linear extension} of $X$.

\section{The reduction of linear extensions of MCQs to MCQ Alexander pairs}
\label{sec:The reduction of linear extensions of MCQs to MCQ Alexander pairs}

In this section, 
we see that 
any quadruple of maps satisfying the conditions \eqref{eq:0-i}--\eqref{eq:4-iii} 
can be reduced to some MCQ Alexander pair.
More precisely, 
any linear extension of an MCQ can be realized by some MCQ Alexander pair up to isomorphism.

Let $X=\bigsqcup_{\lambda \in \Lambda}G_\lambda$ be an MCQ and $R$ be a ring.
Let $(f_1,f_2,f_3,f_4)$ and $(g_1,g_2,g_3,g_4)$ be quadruples of maps satisfying the conditions \eqref{eq:0-i}--\eqref{eq:4-iii}.
Then we write $(f_1,f_2,f_3,f_4) \sim (g_1,g_2,g_3,g_4)$ 
if there exists a map $h:X \to R^\times$ satisfying the following conditions:

\begin{itemize}
\item
For any $x,y \in X$, 
\begin{align*}
h(x \tri y)f_1(x,y)= g_1(x,y)h(x),\\
h(x \tri y)f_2(x,y)=g_2(x,y)h(y).
\end{align*}

\item
For any $a,b \in G_\lambda$, 
\begin{align*}
h(ab)f_3(a,b)=g_3(a,b)h(a),\\
h(ab)f_4(a,b)=g_4(a,b)h(b).
\end{align*}
\end{itemize}
Then $\sim$ is an equivalence relation 
on the set of all quadruples of maps satisfying the conditions \eqref{eq:0-i}--\eqref{eq:4-iii}.
We often write $(f_1,f_2,f_3,f_4) \sim_h (g_1,g_2,g_3,g_4)$ to specify~$h$.
This equivalence relation gives 
an isomorphic linear extensions of MCQs 
as seen in the following proposition.

\begin{proposition}\label{prop:cohomologous}
Let $X=\bigsqcup_{\lambda \in \Lambda}G_\lambda$ be an MCQ, 
$R$ a ring and $M$ a left $R$-module.
Let $(f_1,f_2,f_3,f_4)$ and $(g_1,g_2,g_3,g_4)$ be quadruples of maps satisfying the conditions \eqref{eq:0-i}--\eqref{eq:4-iii}.
If $(f_1,f_2,f_3,f_4)\sim (g_1,g_2,g_3,g_4)$, 
then there exists an MCQ isomorphism $\phi : \widetilde{X}(f_1,f_2,f_3,f_4) \to \widetilde{X}(g_1,g_2,g_3,g_4)$ 
such that $\pr_X \circ \phi=\pr_X$ 
for the projection $\pr_X: \bigsqcup_{\lambda \in \Lambda}(G_\lambda \times M) \to X$ sending $(x,u)$ to $x$.
\end{proposition}

\begin{proof}
Assume that $(f_1,f_2,f_3,f_4)\sim_h (g_1,g_2,g_3,g_4)$ for some map $h:X \to R^\times$.
Let $\phi$ be the map from $\widetilde{X}(f_1,f_2,f_3,f_4)$ to $\widetilde{X}(g_1,g_2,g_3,g_4)$ sending $(x,u)$ to $(x,h(x)u)$.
It is easy to see that 
$\phi$ is an MCQ isomorphism 
and that $\pr_X \circ \phi=\pr_X$.
\end{proof}


\begin{lemma}\label{simplification of a quadruple}
Let $X=\bigsqcup_{\lambda \in \Lambda}G_\lambda$ be an MCQ 
and $R$ a ring.
Let $f_1,f_2,f_3$ and $f_4$ be maps satisfying the conditions \eqref{eq:0-i}--\eqref{eq:4-iii}.
We define maps $g_1,g_2:X \times X \to R$ and $g_3,g_4:\bigsqcup_{\lambda \in \Lambda}(G_\lambda \times G_\lambda) \to R$ by 
\begin{align*}
&g_1(x,y):=f_1(e_x,y),\\
&g_2(x,y):=f_3(x \tri y, x^{-1}\tri y)f_2(x,y)f_3(e_y,y),\\
&g_3(a,b):=1,\\
&g_4(a,b):=f_1(e_a,a^{-1}).
\end{align*}
Then the following hold.
\begin{enumerate}
\item
The maps $g_1$, $g_2$, $g_3$ and $g_4$ satisfy the conditions \eqref{eq:0-i}--\eqref{eq:4-iii}.
\item
The pair $(g_1,g_2)$ is an MCQ Alexander pair.
\end{enumerate}
\end{lemma}

\begin{proof}
\begin{enumerate}
\item
We prove that $g_1$, $g_2$, $g_3$ and $g_4$ satisfy the conditions \eqref{eq:0-i}--\eqref{eq:4-iii} in five steps.

\begin{enumerate}[Step 1:]
\setcounter{enumii}{-1}
\setlength{\leftskip}{-30pt}
\item
We can easily check that $g_1$, $g_2$, $g_3$ and $g_4$ satisfy the conditions \eqref{eq:0-i}, \eqref{eq:0-ii} and \eqref{eq:0-iii}.
For any $a,b,c \in G_\lambda$, it follows 
\begin{align*}
g_4(ab,c)
=f_1(e_{\lambda},b^{-1}a^{-1})
=f_1(e_{\lambda},a^{-1})f_1(e_{\lambda},b^{-1})
=g_4(a,bc)g_4(b,c),
\end{align*}
where the second equality comes from \eqref{eq:2-ii}.
Hence the maps $g_1$, $g_2$, $g_3$ and $g_4$ satisfy  the condition \eqref{eq:0-iv}.

\item
We can easily check that $g_1$, $g_2$, $g_3$ and $g_4$ satisfy the condition \eqref{eq:1-i}.
For any $a,b \in G_\lambda$, it follows
\begin{align*}
&g_2(a,b)\\
&=f_3(a \tri b,a^{-1} \tri b)f_2(a,b)f_3(e_\lambda,b)\\
&=-f_3(b^{-1}ab,b^{-1}a^{-1}b)f_3(b^{-1},ab)f_4(b^{-1},e_\lambda)f_3(b,b^{-1})f_3(e_\lambda,b)\\
   &\quad +f_3(b^{-1}ab,b^{-1}a^{-1}b)f_4(b^{-1},ab)f_4(a,b)f_3(e_\lambda,b)\\
&=-f_3(b^{-1},b)f_4(b^{-1},e_\lambda)+f_4(b^{-1},b)f_3(ab,b^{-1}a^{-1}b)f_4(a,b)f_3(e_\lambda,b) \\
&=-f_4(b^{-1},b)f_3(e_\lambda,b)+f_4(b^{-1},b)f_3(ab,b^{-1}a^{-1}b)f_4(ab,e_\lambda)f_4(b^{-1},b)f_3(e_\lambda,b) \\
&=-f_1(e_\lambda,b)+f_4(b^{-1},b)f_4(ab,b^{-1}a^{-1}b)f_3(e_\lambda,b^{-1}a^{-1}b)f_4(b^{-1},b)f_3(e_\lambda,b), 
\end{align*}
where the second (resp. third) equality comes from \eqref{eq:1-ii} (resp. \eqref{eq:0-ii} and \eqref{eq:0-iii}), 
and where the fourth (resp. fifth) equality comes from \eqref{eq:0-iii} and \eqref{eq:0-iv} (resp. \eqref{eq:0-iii} and \eqref{eq:1-i}).
On the other hand, 
it follows 
\begin{align*}
&-g_3(b^{-1},ab)g_4(b^{-1},e_\lambda)g_3(b,b^{-1})+g_4(b^{-1},ab)g_4(a,b)\\
&=-f_1(e_\lambda,b)+f_1(e_\lambda,b)f_1(e_\lambda,a^{-1})\\
&=-f_1(e_\lambda,b)+f_1(e_\lambda,b^{-1}a^{-1}b)f_1(e_\lambda,b)\\
&=-f_1(e_\lambda,b)+f_4(b^{-1}ab,b^{-1}a^{-1}b)f_3(e_\lambda,b^{-1}a^{-1}b)f_4(b^{-1},b)f_3(e_\lambda,b)\\
&=-f_1(e_\lambda,b)+f_4(b^{-1},b)f_4(ab,b^{-1}a^{-1}b)f_3(e_\lambda,b^{-1}a^{-1}b)f_4(b^{-1},b)f_3(e_\lambda,b),
\end{align*}
where the second (resp. third) equality comes from \eqref{eq:3-i} (resp. \eqref{eq:1-i}), 
and where the fourth equality comes from \eqref{eq:0-iv}.
Hence we have 
\[g_2(a,b)=-g_3(b^{-1},ab)g_4(b^{-1},e_\lambda)g_3(b,b^{-1})+g_4(b^{-1},ab)g_4(a,b),\] 
which implies that 
the maps $g_1$, $g_2$, $g_3$ and $g_4$ satisfy the condition \eqref{eq:1-ii}.

\item
We can easily check that $g_1$, $g_2$, $g_3$ and $g_4$ satisfy the condition \eqref{eq:2-i}.
For any $x \in X$ and $a,b \in G_{\lambda}$, it follows 
\begin{align*}
g_1(x,ab)
=f_1(e_x,ab)
=f_1(e_x \tri a,b)f_1(e_x,a)
=g_1(x \tri a,b)g_1(x,a),
\end{align*}
where the second equality comes from \eqref{eq:2-ii}.
It follows 
\begin{align*}
g_2(x,ab)g_3(a,b)
&=f_3(x \tri ab,x^{-1} \tri ab)f_2(x,ab)f_3(e_{\lambda},ab) \\
&=f_3(x \tri ab,x^{-1} \tri ab)f_2(x,ab)f_3(a,b)f_3(e_{\lambda},a)\\
&=f_3(x \tri ab,x^{-1} \tri ab)f_1(x \tri a,b)f_2(x,a)f_3(e_{\lambda},a)\\
&=f_1(e_x \tri a,b)f_3(x \tri a, x^{-1} \tri a)f_2(x,a)f_3(e_{\lambda},a)\\
&=g_1(x \tri a,b)g_2(x,a),
\end{align*}
where the second (resp. third) equality comes from \eqref{eq:0-ii} (resp. \eqref{eq:2-iii}), 
and where the fourth equality comes from \eqref{eq:4-i}.
It follows 
\begin{align*}
g_2(x,ab)g_4(a,b)
&=f_3(x\tri ab,x^{-1} \tri ab)f_2(x,ab)f_3(e_\lambda,ab)f_1(e_\lambda,a^{-1})\\
&=f_3(x\tri ab,x^{-1} \tri ab)f_2(x,ab)f_1(ba,a^{-1})f_3(e_\lambda,ba) \\
&=f_3(x\tri ab,x^{-1} \tri ab)f_2(x,ab)f_1(ba,a^{-1})f_3(b,a)f_3(e_\lambda,b)\\
&=f_3(x \tri ab,x^{-1} \tri ab)f_2(x,ab)f_4(a,b)f_3(e_\lambda,b)\\
&=f_3(x \tri ab,x^{-1} \tri ab)f_2(x \tri a,b)f_3(e_\lambda,b)\\
&=g_2(x \tri a,b),
\end{align*}
where the second (resp. third) equality comes from \eqref{eq:4-i} (resp. \eqref{eq:0-ii}), 
and where the fourth (resp. fifth) equality comes from \eqref{eq:1-i} (resp. \eqref{eq:2-iv}).
Hence the maps $g_1$, $g_2$, $g_3$ and $g_4$ satisfy the conditions \eqref{eq:2-ii}, \eqref{eq:2-iii} and \eqref{eq:2-iv}.

\item
For any $x,y,z \in X$, it follows 
\begin{align*}
g_1(x \tri y,z)g_1(x,y)
&=f_1(e_x \tri y,z)f_1(e_x,y)\\
&=f_1(e_x \tri z,y \tri z)f_1(e_x,z)\\
&=g_1(x \tri z,y \tri z)g_1(x,z),
\end{align*}
where the second equality comes from \eqref{eq:3-i}.
It follows 
\begin{align*}
&g_1(x \tri y,z)g_2(x,y)\\
&=f_1(e_x \tri y,z)f_3(x \tri y,x^{-1} \tri y)f_2(x,y)f_3(e_{y},y)\\
&= f_3((x \tri y) \tri z, (x^{-1} \tri y) \tri z)f_1(x \tri y,z)f_2(x,y)f_3(e_{y},y)\\
&= f_3((x \tri y) \tri z, (x^{-1} \tri y) \tri z)f_2(x \tri z,y \tri z)f_1(y,z)f_3(e_{y},y)\\
&= f_3((x \tri y) \tri z, (x^{-1} \tri y) \tri z)f_2(x \tri z,y \tri z) f_3(e_{y}\tri z,y \tri z)f_1(e_{y},z)\\
&=f_3((x \tri z) \tri (y \tri z),(x^{-1} \tri z) \tri (y \tri z))f_2(x \tri z,y \tri z)f_3(e_{y} \tri z,y \tri z)f_1(e_{y},z)\\
&=g_2(x \tri z,y \tri z)g_1(y,z),
\end{align*}
where the second (resp. third) equality comes from \eqref{eq:4-i} (resp. \eqref{eq:3-ii}), 
and where the fourth equality comes from \eqref{eq:4-i}.
It follows 
\begin{align*}
&g_2(x \tri y,z)\\
&=f_3((x \tri y) \tri z, (x^{-1} \tri y) \tri z)f_2(x \tri y,z)f_3(e_{z},z)\\
&=f_3((x \tri y) \tri z, (x^{-1} \tri y) \tri z)f_1(x \tri z,y \tri z)f_2(x,z)f_3(e_{z},z)\\
      &\quad +f_3((x \tri y) \tri z, (x^{-1} \tri y) \tri z)f_2(x \tri z,y \tri z)f_2(y,z)f_3(e_{z},z)\\
&=f_1(e_x \tri z,y \tri z)f_3(x \tri z, x^{-1} \tri z)f_2(x,z)f_3(e_{z},z)\\
     &\quad +f_3((x \tri z) \tri (y \tri z), (x^{-1} \tri z) \tri (y \tri z))f_2(x \tri z,y \tri z)f_2(y,z)f_3(e_{z},z)\\
&=f_1(e_x \tri z,y \tri z)f_3(x \tri z, x^{-1} \tri z)f_2(x,z)f_3(e_{z},z)\\
     &\quad +f_3((x \tri z) \tri (y \tri z), (x^{-1} \tri z) \tri (y \tri z))f_2(x \tri z,y \tri z)f_3(e_{y} \tri z,y \tri z)\\
     &\quad \qquad f_3(y \tri z, y^{-1} \tri z)f_2(y,z)f_3(e_{z},z)\\
&=g_1(x \tri z,y \tri z)g_2(x,z)+g_2(x \tri z,y \tri z)g_2(y,z),
\end{align*}
where the second (resp. third) equality comes from \eqref{eq:3-iii} (resp. \eqref{eq:4-i}).
Hence the maps $g_1$, $g_2$, $g_3$ and $g_4$ satisfy the conditions \eqref{eq:3-i}, \eqref{eq:3-ii} and \eqref{eq:3-iii}.

\item
We can easily check that $g_1$, $g_2$, $g_3$ and $g_4$ satisfy the condition \eqref{eq:4-i}.
For any $a,b \in G_\lambda$ and $x \in X$, it follows 
\begin{align*}
g_1(ab,x)g_4(a,b)
&=f_1(e_\lambda,x)f_1(e_\lambda,a^{-1})\\
&=f_1(e_\lambda \tri x,a^{-1} \tri x)f_1(e_\lambda,x)\\
&=g_4(a \tri x, b \tri x)g_1(b,x),
\end{align*}
where the second equality comes from \eqref{eq:3-i}.
It follows 
\begin{align*}
&g_2(ab,x)\\
&=f_3(ab \tri x, b^{-1}a^{-1} \tri x)f_2(ab,x)f_3(e_{x},x)\\
&=f_3(ab \tri x, b^{-1}a^{-1} \tri x)f_3(a \tri x,b \tri x)f_2(a,x)f_3(e_{x},x)\\
     & \quad +f_3(ab \tri x, b^{-1}a^{-1} \tri x)f_4(a \tri x,b \tri x)f_2(b,x)f_3(e_{x},x) \\
&=f_3(a \tri x,a^{-1} \tri x)f_2(a,x)f_3(e_{x},x)\\
     & \quad +f_4(a \tri x,a^{-1} \tri x)f_3(b \tri x,b^{-1}a^{-1} \tri x)f_2(b,x)f_3(e_{x},x)\\
&=f_3(a \tri x,a^{-1} \tri x)f_2(a,x)f_3(e_{x},x)\\
     & \quad +f_4(a \tri x,a^{-1} \tri x)f_3(e_\lambda \tri x,a^{-1} \tri x)f_3(b \tri x,b^{-1} \tri x) f_2(b,x)f_3(e_{x},x)\\
&=f_3(a \tri x,a^{-1} \tri x)f_2(a,x)f_3(e_{x},x)\\
     & \quad +f_1(e_\lambda \tri x,a^{-1} \tri x)f_3(b \tri x,b^{-1} \tri x)f_2(b,x)f_3(e_{x},x)\\ 
&=g_3(a \tri x,b \tri x)g_2(a,x)+g_4(a \tri x,b \tri x)g_2(b,x),
\end{align*}
where the second (resp. third) equality comes from \eqref{eq:4-iii} (resp. \eqref{eq:0-ii} and \eqref{eq:0-iii}), 
and where the fourth (resp. fifth) equality comes from \eqref{eq:0-ii} (resp. \eqref{eq:1-i}).
Hence the maps $g_1$, $g_2$, $g_3$ and $g_4$ satisfy the conditions \eqref{eq:4-ii} and \eqref{eq:4-iii}.
\end{enumerate}

This completes the proof.

\item

We show that 
the pair $(g_1,g_2)$ is an MCQ Alexander pair.
Here, we remark that 
the maps $g_1$, $g_2$, $g_3$ and $g_4$ satisfy the conditions \eqref{eq:0-i}--\eqref{eq:4-iii}, 
and that $g_4(a,b)=g_1(a,a^{-1})$ for any $a,b \in G_\lambda$.
We also remark that 
$g_1(a,x)=g_1(b,x)$ for any $a,b \in G_\lambda$ and $x \in X$.

For any $a,b \in G_\lambda$,
it follows
\begin{align*}
g_2(a,b)
&=-g_1(b^{-1},b)+ g_1(b^{-1},b)g_1(a, a^{-1})\\
&=-g_1(a,b)+ g_1(a,b)g_1(a, a^{-1})\\
&=-g_1(a,b)+ g_1(a,a^{-1}b), 
\end{align*}
where the first (resp. third) equality comes from \eqref{eq:1-ii} (resp. \eqref{eq:2-ii}).
Hence we have 
$g_1(a,b)+g_2(a,b)=g_1(a,a^{-1}b)$.

For any $a,b \in G_\lambda$ and $x \in X$, 
we can easily check that $g_1(a,x)=g_1(b,x)$, 
and it follows 
\begin{align*}
g_2(ab,x)
&=g_2(a,x)+g_1(a \tri x,a^{-1} \tri x)g_2(b,x)\\
&=g_2(a,x)+g_1(b \tri x,a^{-1} \tri x)g_2(b,x),
\end{align*}
where the first equality comes from \eqref{eq:4-iii}.

For any $x \in X$ and $a,b \in G_{\lambda}$,
we can easily check that 
\begin{align*}
&g_1(x,e_{\lambda})=1,\\
&g_1(x,ab)=g_1(x \tri a,b)g_1(x,a),\\
&g_2(x,ab)=g_1(x \tri a,b)g_2(x,a)
\end{align*}
by \eqref{eq:2-i}, \eqref{eq:2-ii} and \eqref{eq:2-iii}. 

For any $x,y,z \in X$, 
we can easily check that 
\begin{align*}
&g_1(x \tri y,z)g_1(x,y)=g_1(x \tri z,y \tri z)g_1(x,z),\\
&g_1(x \tri y,z)g_2(x,y)=g_2(x \tri z,y \tri z)g_1(y,z),\\
&g_2(x \tri y,z)=g_1(x \tri z,y \tri z)g_2(x,z)+g_2(x \tri z,y \tri z)g_2(y,z)
\end{align*}
by \eqref{eq:3-i}, \eqref{eq:3-ii} and \eqref{eq:3-iii}. 

Therefore the pair $(g_1,g_2)$ is an MCQ Alexander pair.
\end{enumerate}
\end{proof}

\begin{theorem}
Let $X=\bigsqcup_{\lambda \in \Lambda}G_\lambda$ be an MCQ, 
$R$ a ring and $M$ a left $R$-module.
For any quadruple $(f_1,f_2,f_3,f_4)$ of maps satisfying the conditions \eqref{eq:0-i}--\eqref{eq:4-iii}, 
there exists an MCQ Alexander pair $(g_1,g_2)$ 
such that $\widetilde{X}(f_1,f_2,f_3,f_4) \cong \widetilde{X}(g_1,g_2)$.
\end{theorem}

\begin{proof}
Let $(f_1,f_2,f_3,f_4)$ be a quadruple of maps satisfying the conditions \eqref{eq:0-i}--\eqref{eq:4-iii}, 
and let $g_1,g_2:X \times X \to R$ and $g_3,g_4:\bigsqcup_{\lambda \in \Lambda}(G_\lambda \times G_\lambda) \to R$ be maps 
defined by 
\begin{align*}
&g_1(x,y):=f_1(e_x,y),\\
&g_2(x,y):=f_3(x \tri y, x^{-1}\tri y)f_2(x,y)f_3(e_y,y),\\
&g_3(a,b):=1,\\
&g_4(a,b):=f_1(e_a,a^{-1}).
\end{align*}
By Lemma~\ref{simplification of a quadruple} (1), 
the maps $g_1,g_2,g_3$ and $g_4$ satisfy the conditions \eqref{eq:0-i}--\eqref{eq:4-iii}.
We define the map $h:X \to R^{\times}$ by $h(x):=f_3(x,x^{-1})$.
Then for any $x,y \in X$, it follows 
\begin{align*}
h(x \tri y)f_1(x,y)
&=f_3(x \tri y, x^{-1} \tri y)f_1(x,y)\\
&=f_1(e_x,y)f_3(x,x^{-1})\\
&=g_1(x,y)h(x),
\end{align*}
where the second equality comes from \eqref{eq:4-i}, 
and
\begin{align*}
h(x \tri y)f_2(x,y)
=f_3(x \tri y, x^{-1} \tri y)f_2(x,y)f_3(e_y,y)f_3(y,y^{-1})
=g_2(x,y)h(y).
\end{align*}
For any $a,b \in G_\lambda$, it follows 
\begin{align*}
h(ab)f_3(a,b)
=f_3(ab, b^{-1}a^{-1})f_3(a,b)
=f_3(a,a^{-1})
=g_3(a,b)h(a),
\end{align*}
where the second equality comes from \eqref{eq:0-ii}, 
and
\begin{align*}
h(ab)f_4(a,b)
&=f_3(ab,b^{-1}a^{-1})f_4(a,b)f_3(b,b^{-1})^{-1}f_3(b,b^{-1})\\
&=f_4(a,a^{-1})f_3(b,b^{-1}a^{-1})f_3(e_\lambda,b)f_3(b,b^{-1})\\
&=f_4(a,a^{-1})f_3(e_\lambda,a^{-1})f_3(b,b^{-1})\\
&=f_1(e_\lambda,a^{-1})f_3(b,b^{-1}) \\
&=g_4(a,b)h(b),
\end{align*}
where the second (resp. third) equality comes from \eqref{eq:0-iii} (resp. \eqref{eq:0-ii}), 
and where the fourth equality comes from \eqref{eq:1-i}.
Hence we have $(f_1,f_2,f_3,f_4)\sim_h (g_1,g_2,g_3,g_4)$, 
which implies $\widetilde{X}(f_1,f_2,f_3,f_4) \cong \widetilde{X}(g_1,g_2,g_3,g_4)$ 
by Proposition~\ref{prop:cohomologous}.
By Lemma~\ref{simplification of a quadruple} (2), 
$(g_1,g_2)$ is an MCQ Alexander pair.
Therefore we obtain that 
$\widetilde{X}(g_1,g_2,g_3,g_4) = \widetilde{X}(g_1,g_2)$ by the definitions.
This completes the proof.
\end{proof}

\section*{Acknowledgment}
The author would like to thank 
Atsushi Ishii and Shosaku Matsuzaki
for valuable discussions 
and making suggestions 
for improvement.
The author was supported by JSPS KAKENHI Grant Number 18J10105.


\end{document}